\documentclass[12pt,a4paper]{amsart}
\usepackage{geometry}
\usepackage{amsmath,amssymb,amsfonts,amscd,amsthm,wasysym,color,enumitem,indentfirst,graphicx,booktabs}
\usepackage[bookmarksnumbered,colorlinks,linktocpage,plainpages]{hyperref}
\hypersetup{pdfstartview={FitH}}
\usepackage{mathtools}

\numberwithin{equation}{section}
\newtheorem{theorem}{Theorem}[section]

\newtheorem{lemma}[theorem]{Lemma}

\theoremstyle{definition}
\newtheorem{definition}[theorem]{Definition}

\newtheorem{remark}[theorem]{Remark}

\usepackage{tikz}

\newtheorem{conjecture}{Conjecture}

\allowdisplaybreaks[4]

\makeatletter
\@namedef{subjclassname@2020}{\textup{2020} Mathematics Subject Classification}
\makeatother

\def\abs#1{\left|#1\right|}
\date{}
\def\A{\mathbb{A}}
\def\R{\mathbb{R}}
\def\C{\mathbb{C}}
\def\O{\mathbb{O}}

\def\HP{H^p(\mathbb{R}^8_+, \mathbb{O})}
\def\HN#1{\|#1\|_{H^p(\mathbb{R}^8_+, \mathbb{O})}}
\def\HNC#1{\|#1\|_{H^p(\mathbb{C}_+)}}
\def\P{\mathcal{P}(n)}
\newcommand{\bfe} {{\mathbf e}}
\begin{document}
\title[Gilbert's conjecture ]{Gilbert's conjecture and
A new way to octonionic analytic functions from the clifford analysis}
\date{\today}

\author{Yong Li}
\address{School of Mathematics and Statistics, Anhui Normal University, Wuhu 241002, Anhui, People's Republic of China}
\email{leeey@ahnu.edu.cn}

\thanks{The author was partially supported by Anhui Province Scientific Research Compilation Plan Project 2022AH050175.}

\subjclass[2020]{Primary 30A05, 30G35, 30H10 ; Secondary 42B30, 42B35}
\keywords{Clifford analysis, Hardy space, Gilbert's conjecture, Octonionic analytic functions, Octonionic Hardy space}


\begin{abstract}
In this article we will give a affirmative answer to Gilbert's conjecture on  Hardy spaces of Clifford analytic functions in upper half-space of $\R^8$.
It depends on a explicit construction of Spinor space $\mathcal{R}_8$ and Clifford algebra $Cl_8$ by octonion algbra. What's more , it gives us an associative way
to octonionic analytic function theory. And the similar question has been discussed in Octonionic Hardy space in upper-half space, some classical results about octonionic analytic functions have been reformulated, too.
\end{abstract}
\maketitle
\tableofcontents
\section{Introduction}
The classical $H^p$ theory - the study of spaces of functions analytic in the upper half-plane $\mathbb{C}_+=\{z\in \mathbb{C}: Imz>0\}$ started in 1915, when G.H. Hardy considered the  analytic functions in $\mathbb{C}_+$, with
$$\HNC{f}=\sup_{t>0} \left(\int_{\mathbb{R}}{\abs{F(x+it)}^p}\mathrm{d}x\right)^{\frac{1}{p}}<\infty.$$
With some great mathematicians efforts, the theory has developed into a elegant theory in mathematics.

A classical result about $H^p$ theory says, $H^p$ function can be determined by its real part of its non-tangent boundary value. We have following routine to get this result.
\begin{theorem}
  Suppose $F\in H^p(\mathbb{C}_+), p>1$. Then there is a function $f\in L^p(\R,\C)$ such that
  \begin{enumerate}
    \item $\lim\limits_{z\to x \  n.t} F(z)=f(x)$ exists for almost $x\in \R$.
    \item $\lim\limits_{t\to 0}\displaystyle\int_{-\infty}^{+\infty}\abs{F(x+it)-f(x)}^p\mathrm{d}x=0.$
  \end{enumerate}
\end{theorem}

Conversely, for a $f\in L^p(\R,\C), p\ge 1$ we can consider its Cauchy integral $\mathcal{C}(f)$:
$$\mathcal{C}(f)=\frac{1}{2\pi i}\int_{-\infty}^{+\infty}\frac{f(t)}{t-z}\mathrm{d} t, \quad \quad z\in \C\setminus \R.$$
$\mathcal{C}(f)$ is analytic function on $\mathbb{C}_+$ with boundary value $\dfrac{1}{2}(f+iHf)$, where $H$ is the Hilbert transform operator defined by:
$$H(f)(x)=p.v.\frac{1}{\pi}\int_{-\infty}^{+\infty}\frac{f(y)}{x-y}\mathrm{d}y.$$
So we can view $H^p(\mathbb{C}_+)$ as $f\in L^p(\R,\C), f=iHf$. Actually we can identify  $H^p(\mathbb{C}_+)$ with $ L^p(\R,\R)$ by take real parts of those $f\in L^p(\R,\C)$ and $ f=iHf$.
Or roughly speaking, \begin{equation}\label{eq：complex}
           H^p(\mathbb{C}_+) =
           \begin{dcases}
           L^p(\R, \R), & \quad p>1\\
           \{f\in L^1(\R, \R): Hf\in L^1(\R, \R)\}, &\quad p=1.
           \end{dcases}
         \end{equation}

Are these things can be bringed into Clifford cases or octonionic case, this is a conjecture recorded in \cite[P.140 Conjecture 7.23]{Gilbert}.

\subsection{Clifford module-valued Hardy space}
Under the development of Clifford analysis, the Clifford Hardy space theory has be established completely similarly.
We recall some basics and the Gilbert's conjecture.

For a Clifford module $\mathfrak{H}$ (see Definition \ref{def:CM}), we said $F\in H^p(\R^n_+, \mathfrak{H})$ , if $F$ is a Clifford analytic function in upper-half space
$\R^n_+=\{(t, \underline{x}): t>0,\underline{x}\in \R^{n-1}\}$ and
$$\|F\|_{H^p(\R^n_+, \mathfrak{H})}=\sup_{t>0}\left(\int_{\mathbb{R}^{n-1}}\abs{F(t, \underline{x})}^p\mathrm{d} x\right)^{\frac{1}{p}}<\infty.$$

 Firstly, analogue of classical case, we shall show that for $p>\frac{n-2}{n-1}$, every $H^p\left(\mathbb{R}_{+}^n, \mathfrak{H}\right)$ function has almost-everywhere non-tangential limits.
\begin{theorem}\cite[P.120 Theorem 5.4]{Gilbert}\label{thm:bdv}
  Suppose $F \in H^p\left(\mathbb{R}_{+}^n, \mathfrak{H}\right), p>\frac{n-2}{n-1}$. Then there is a function $f \in$ $L^p\left(\mathbb{R}^{n-1}, \mathfrak{H}\right)$ such that
  \begin{enumerate}
    \item $\lim \limits_{z \rightarrow x  n.t.}$ $F(z)=f(x)$ exists for almost all $x \in \mathbb{R}^{n-1}$,
    \item $\lim \limits_{t \rightarrow 0}\displaystyle \int_{\mathbb{R}^{n-1}}|F(x, t)-f(x)|^p d x=0$.
  \end{enumerate}
\end{theorem}

Then, for $p \geq 1$, we shall give a boundary integral characterization of $H^p\left(\mathbb{R}_{+}^n, \mathfrak{H}\right)$, whereby the Hardy space may be identified with a subspace of $L^p\left(\mathbb{R}^{n-1}, \mathfrak{H}\right)$ on the boundary.
More precisely, $f \in L^p\left(\mathbb{R}^{n-1}, \mathfrak{H}\right)$. We define the Cauchy integral of $f$ on $\mathbb{R}^{n-1}$ by setting
\begin{equation}\label{eq:Cauchy}
  C f(z)=\frac{1}{\omega_n} \int_{\mathbb{R}^{n-1}} \frac{u-z}{|u-z|^n} \bfe_0 f(u) d u
\end{equation}
for $z=(x, t) \in \mathbb{R}^n \backslash \mathbb{R}^{n-1}$. $C f$ is clearly analytic on $\mathbb{R}_{+}^n$, and
\begin{align*}
   C f(z)&=\frac{1}{2}\left(P_t * f\right)(x)+\frac{1}{2} \sum_{j=1}^{n-1}  \bfe_0\bfe_j\left(Q_t^{(j)} * f\right)(x)  \\
   & =\left\{P_t * \frac{1}{2}\left(I+\bfe_0\mathcal{H} \right) f\right\}(x), \quad z=(x, t) \in \mathbb{R}_{+}^n.
\end{align*}
Where $\mathcal{H}=\sum_{j=1}^{n-1} \bfe_j R_j$, and $R_j$ is the $j$ th Riesz transform, given by
$$
R_j g(x)=\frac{2}{\omega_n} \int_{\mathbb{R}^{n-1}} \frac{x_j-u_j}{|x-u|^n} g(u) \mathrm{d} u, \quad \quad 1 \leq j \leq n-1
$$
and
\begin{align*}
  P_t(x)=&\frac{2}{\omega_n} \frac{t}{\left[t^2+|x|^2\right]^{n / 2}}=\frac{2}{\omega_n} \frac{1}{t^{n-1}} \frac{1}{\left[1+|x / t|^2\right]^{n / 2}}  \\
  Q_t^{(j)}(x)=& \frac{2}{\omega_n} \frac{x_j}{\left[t^2+|x|^2\right]^{n / 2}}=\frac{2}{\omega_n} \frac{1}{t^{n-1}} \frac{x_j / t}{\left[1+|x / t|^2\right]^{n / 2}}, \quad j=1,\cdots, n-1,
\end{align*}
$$$$
is the Poisson kernel  and $j$ th conjugate Poisson kernel for $\mathbb{R}^{n-1}$,
the following result is a straightforward consequence of the properties of the Poisson kernel.

\begin{theorem}\cite[P.122 Theorem 5.16]{Gilbert}\label{thm:Hp}
Suppose that either (i) $1<p<\infty$ and $f \in L^p\left(\mathbb{R}^{n-1}, \mathfrak{H}\right)$, or (ii) $p=$ 1 and $f, \mathcal{H} f \in L^1\left(\mathbb{R}^{n-1}, \mathfrak{H}\right)$. Then $C f \in H^p\left(\mathbb{R}_{+}^n, \mathfrak{H}\right)$, and
$$
\lim _{z \rightarrow x, n.t} C f(z)=\frac{1}{2}\left(I+\bfe_0\mathcal{H} \right) f(x)
$$
for almost all $x \in \mathbb{R}^{n-1}$.

Conversely, if $1 \leq p<\infty$ and suppose $F \in H^p\left(\mathbb{R}_{+}^n, \mathfrak{H}\right)$. Then $F=C f$, where $f$ is the almost-everywhere non-tangential limit of $F$ given by Theorem \ref{thm:bdv}.
\end{theorem}



Theorem \ref{thm:Hp} says, in effect, that, for $1<p<\infty, H^p\left(\mathbb{R}_{+}^n, \mathfrak{H}\right)$ is precisely the image of $L^p\left(\mathbb{R}^{n-1}, \mathfrak{H}\right)$ under the Cauchy integral operator, while $H^1\left(\mathbb{R}_{+}^n, \mathfrak{H}\right)$ is the image under $C$ of the set of all $f \in$ $L^1\left(\mathbb{R}^{n-1}, \mathfrak{H}\right)$ for which $\mathcal{H} f \in L^1\left(\mathbb{R}^{n-1}, \mathfrak{H}\right)$. Moreover, for $1 \leq p<\infty$, $f \in L^p\left(\mathbb{R}^{n-1}, \mathfrak{H}\right)$ arises as the non-tangential boundary value of an $H^p\left(\mathbb{R}_{+}^n, \mathfrak{H}\right)$-function if and only if $f=\bfe_0\mathcal{H}  f \in L^p\left(\mathbb{R}^{n-1}, \mathfrak{H}\right)$.

Thus for $p>1$, in the special case of $\mathfrak{H}=Cl_n, H^p\left(\mathbb{R}_{+}^n, Cl_n\right)$ is effectively isomorphic to $L^p\left(\mathbb{R}^{n-1}, Cl_{n-1}\right)$, while $H^1\left(\mathbb{R}_{+}^n Cl_n\right)$ is isomorphic to a subspace of $L^1\left(\mathbb{R}^{n-1}, Cl_{n-1}\right)$.

And in the special case of $\mathfrak{H}=Cl_n, H^p\left(\mathbb{R}_{+}^n, Cl_n\right)$ has a similar description like \eqref{eq：complex}:
\begin{equation}\label{eq：CLn}
           H^p(\mathbb{C}_+) =
           \begin{dcases}
           L^p(\R^{n-1}, Cl_{n-1}), & \quad p>1\\
           \{f\in L^1(\R^{n-1}, Cl_{n-1}): \bfe_0\mathcal{H}f\in L^1(\R, Cl_{n-1})\}, &\quad p=1.
           \end{dcases}
         \end{equation}



In general cases, Gilbert makes following conjecture in his book .
\begin{conjecture}\cite[P.140 Conjecture 7.23]{Gilbert}
  For every Clifford module $\mathcal{H}$ there exists $\eta \in \mathbb{R}^n$ with $\eta^2=-1$ and a subspace $\mathfrak{H}_0$ of $\mathfrak{H}$, such that
  \begin{enumerate}
    \item $\mathcal{H}$ admits the splitting $\mathfrak{H}=\mathfrak{H}_0+\eta\mathfrak{H}_0$,
    \item the Cauchy integral operator $C_M$ is an isomorphism from $L^p(\partial M, \mathfrak{H}_0 )$ onto $H^p(M,\mathfrak{H})$ for all $p>p_0$,
    \item if $Tan$ denotes the projection of $\mathfrak{H}$ onto$\mathcal{H}_0$, then the boundary operator mapping $F\to Tan(F^+)$ is continuous from $H^p(M,\mathfrak{H})$ onto $L^p(\partial M, \mathfrak{H}_0 )$ for all $p>p_0$.
  \end{enumerate}
\end{conjecture}
 We will give a affirmative answer in the case  Hardy spaces of Clifford analytic functions in upper half-space of $\R^8$.
 \begin{theorem}\label{thm:main1}
   For every Clifford module $\mathfrak{H}$ there exists $\eta \in \mathbb{R}^8$ with $\eta^2=-1$ and a subspace $\mathfrak{H}_0$ of $\mathfrak{H}$, such that
  \begin{enumerate}
    \item $\mathfrak{H}$ admits the splitting $\mathfrak{H}=\mathfrak{H}_0\oplus \eta\mathfrak{H}_0$,
    \item the Cauchy integral operator $C$ is an isomorphism from $L^p(\R^7, \mathfrak{H}_0 )$ onto $H^p(\R^8_+,\mathfrak{H})$ for all $p>1$,
    \item if $Tan$ denotes the projection of $\mathfrak{H}$ onto $\mathfrak{H}_0$, then the boundary operator mapping $F\to Tan(F^+)$ is continuous from $H^p(\R^8_+,\mathfrak{H})$ onto $L^p(\R^7, \mathfrak{H}_0 )$ for all $p>1$.
   \end{enumerate}
 \end{theorem}

\subsection{Octonionic Hardy space}
When we review the material we used in Clifford case, the Cauchy integral involves $\bfe_0\bfe_j$ only, that means we can consider the Spin(8) module or $Cl_{7}$ module. It's  known the octonion algebra can be realized as a $Cl_{7}$ module\cite{Harvey90,HLR21}.

What will happen if we take Cauchy integral of a octonion-valued function ?

Under the natural question and with the help of  the previous progresses, we find a way from Clifford analysis to octonion analytic function theory. And octonionic Hardy space theory can be built easily,  in which case   Theorem \ref{thm:bdv} and Theorem \ref{thm:Hp} still holds. Thus , it's very natural to ask the same question  to octonion case, but  we have the following totally different answer .

\begin{theorem}\label{thm:main2}
  There don't exist  two proper subspaces $\mathfrak{H}_0, \mathfrak{H}_1$ of $\O$, satisfy the following all three statements:
  \begin{enumerate}
    \item $\O$ admits the splitting $\mathfrak{H}=\mathfrak{H}_0\oplus \mathfrak{H}_1$,
    \item the Cauchy integral operator $C_{\O}$ is an isomorphism from $L^p(\R^7, \mathfrak{H}_0 )$ onto $H^p(\R^8_+,\O)$ for all $p>1$,
    \item if $Tan$ denotes the projection of $\O$ onto$\mathfrak{H}_0$, then the boundary operator mapping $F\to Tan(F^+)$ is continuous from $H^p(\R^8_+,\O)$ onto $L^p(\R^7, \mathfrak{H}_0 )$ for all $p>1$.
   \end{enumerate}
 \end{theorem}

Our paper is organized as follows. Section 2 is devoted to recalling some basics for Clifford algebras and octonion algebra. In Section 3, the most important realization of $Cl_8$ and spinor space $\mathcal{R}_8$ are given, which connects octonion algebra and $\mathcal{R}_8$ closely . A new way from Clifford analysis to Octonionic analytic functions theory has been introduced in Section 4, and octonionic Hardy space in upper half space has been built by Clifford Hardy spaces.
And at last section, we will give the proofs of Theorem \ref{thm:main1} and Theorem \ref{thm:main2}.

\section{Preliminaries}
\subsection{Some basic properties of Universal Clifford algebra $Cl_n$ over $\mathbb{R}^n$}

We don't want to discuss the properties of universal Clifford algebra $Cl_n$ too much here, but the following notions and conventions are needed.
For a detailed discussion on Clifford algebra, we refer the readers to \cite{Atiyah,Gilbert,F.S,BDS82}.

\begin{definition}
  Let $\mathbb{A}$ be a associative algebra over $\mathbb{R}$ with identity 1,$v:\R^n\rightarrow \A $ is a $\R$-linear embedding.The pair $(\A,v)$ is said to be a universal Clifford algebra $Cl_n$ over $\R^n$,if:
  \begin{enumerate}[label=(\arabic*)]
    \item $A$ is generated as an algebra by $\{v(x):x\in \R^n\}$ and $\{\lambda1,\lambda\in\R\}$,
    \item $(v(x))^2=-\abs{x}^2,\forall x\in \R^n$
    \item $dim_{\R}\A=2^n$
  \end{enumerate}
\end{definition}

\textbf{Some notations and conventions}
\begin{itemize}
 \item\label{g_i} Let $\{\bfe_{i-1}=(\underbrace{0,\cdots,0}_{i-1},1,0,\cdots,0),i=1,\cdots,n\}$ be the orthogonal normalized basis of $\R^n$, and we don't distinguish $\bfe_i\in \mathbb{R}^n$ and $v(\bfe_i)\in Cl_n$ ;
  \item $\mathbf{x}=\sum_{i=0}^{n-1}x_i\bfe_i \in \R^n,\abs{x}^2=\sum_{i=0}^{n-1}x_i^2$;

  \item Let$\mathcal{P}(n)$ be the set of subset of $\{0,\cdots,n-1\}$, $\forall \alpha\in \P ,\alpha\neq \emptyset,\alpha=\{\alpha_1\cdots,\alpha_k\},0\le \alpha_1<\cdots<\alpha_k\le n-1$, denote $\bfe_{\alpha}=\bfe_{\alpha_1}\cdots \bfe_{\alpha_k}$, and $\bfe_{\emptyset}=1$.
  \item $Cl_n$ is $\R$ linearly generated  by $\{\bfe_\alpha,\alpha\in \P\}$;
    \item $\bfe_i\bfe_j+\bfe_j\bfe_i=-2\delta_{ij},\forall i,j=0,\cdots,n-1$.
    \item $x=\sum_{A}x_A \bfe_A \in Cl_n$, the Clifford conjugation $x^\star=\sum_{A}x_A \bfe_A^\star $, where $\bfe_A^\star=(-1)^{\frac{|A|(|A|+1)}{2}}\bfe_A.$
\end{itemize}

\begin{definition}\label{def:CM}
  A finite dimensional real Hilbert space $\mathcal{H}$ is said to  be $Cl_n$ module when there exist skew-adjoint real linear operator $T_1,\cdots, T_n$ such that
  $$T_jT_k+T_kT_j=-2\delta_{jk}id,\quad \quad (1\le j,k \le n).$$
\end{definition}

 Among all $Cl_n$ module, there is a class of most important  $Cl_n$ module named real spinor space $\mathcal{R}_n$. When $n\neq 4l+3$, it can be characterized as $Cl_n=End_{\R}(\mathcal{R}_n)$\cite[P.59 (7.42)]{Gilbert}, where $End_{\R}(\mathcal{R}_n)$ is the real linear operator algebra from $\mathcal{R}_n$ to $\mathcal{R}_n$.
 \subsection{Octonion algebra}

The  algebra of   octonions   $\O$  is  a non-commutative, non-associative,  normed 8 dimensional division algebra over  $\R$. Comparing to the Clifford algebra, the following notions and conventions are needed. We refer to \cite{Baez,CS03} for a detailed discussion on octonion algebra.

\textbf{Some notations and conventions again}
\begin{itemize}
\item Let     $\mathbf{e_0}=1,\bfe_1 \ldots,\mathbf{e_7}$ be its natural basis,  we have $$\mathbf{e_i}\mathbf{e_j}+\mathbf{e_j}\mathbf{e_i}=-2\delta_{ij},\quad i,j=1,\ldots,7.$$Noticed that in $\R^8$, we can have three different means of $\bfe_i$, a vector in $\R^8$, a Clifford number in $Cl_8$, and a octonion number in $\mathbb{O}$, we don't distinguish them if they don't make confuses.
\item $x=x_0+\sum_{i=1}^7x_i\mathbf{e_i}\in \O,\quad x_i\in\R$ we define octonion conjugation by $\overline{x}=x_0-\sum_{i=1}^7x_i\mathbf{e_i} $, and real part operator $Re x=x_0$, $\O$ is a 8-dimensional Euclidean space, under inner product $(p,q)=Re(p\bar{q}), \quad p,q\in \O$.
\item We shall use the associator, defined by
$$[a, b, c]=(ab)c-a(bc).$$
It is well-known that the associator of octonion is alternative, i.e.,
$$[a, b, c]=-[a, c, b]=-[b,a,c]=-[\overline{a}, b, c].$$

\item The full multiplication table is conveniently encoded in  the Fano plane (see Figure 1 and \cite{Baez}).
In the Fano plane, the vertices are labeled by  $\mathbf{e_1},\ldots, \mathbf{e_7}$.
Each of the 7 oriented lines gives a quaternionic triple. The
product of any two imaginary units is given by the third unit on the unique line
connecting them, with the sign determined by the relative orientation.

\begin{figure}[ht]
\centering
  \includegraphics[width=6cm]{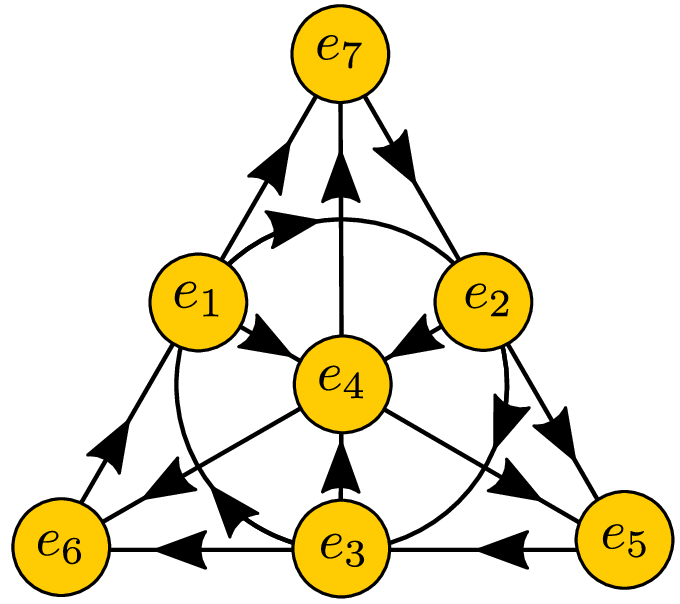}
  \caption{Fano plane}
  \label{fig:1}
\end{figure}

\end{itemize}

%
%

\section{An explicit realization of Clifford algebra $Cl_8$ and the spinor space $\mathcal{R}_8$}
Let $V=\R^8$ be the real $8$ dimensional Euclidean space. We give a embedding  $A:V\to End_{\R}(\mathbb{O}\oplus \mathbb{O})$ by
$$A(q)= \left(\begin{matrix}
 0 & L_q\\
 -L_{\bar{q}} & 0
\end{matrix} \right).
$$
Where the $L_q\in End_{\R}(\mathbb{O})$ is the left multiply operator defined by $$L_q:\mathbb{O}\to \mathbb{O}, L_q(p)=qp,$$ for any $p\in \mathbb{O}$.

\begin{lemma}
	$\forall p,q \in \mathbb{O}$, we have
	$$L_pL_{\bar{q}}+L_qL_{\bar{p}}=2(p,q)id.$$
\end{lemma}
\begin{proof}
	This simple fact just follows that for any octonion $z$, we have
	\begin{align*}
		&\left(L_pL_{\bar{q}}+L_qL_{\bar{p}}\right)(z)=p(\bar{q}z)+q(\bar{p}z)\\
		=&-[p, \bar{q}, z]+(p\bar{q})z-[q, \bar{p}, z]+(q\bar{p})z=(p\bar{q}+q\bar{p})z\\
=&2Re(p\overline{q})=2(p,q)z.
	\end{align*}
\end{proof}

\begin{theorem}[\cite{Harvey90}]\label{realization}
The map $A:V\to End_{\R}(\mathbb{O}\oplus \mathbb{O})$ gives a realization of Clifford algebra $Cl_8$ , thus $\mathbb{O}\oplus \mathbb{O}$ can be considered as the real spinor space $\mathcal{R}_8$.
\end{theorem}
\begin{proof}
	Actually we have
	\begin{align*}
		A(p)A(q)+A(q)A(p)=&\left(\begin{matrix}
 0 & L_p\\
 -L_{\bar{p}} & 0
\end{matrix} \right)\left(\begin{matrix}
 0 & L_q\\
 -L_{\bar{q}} & 0
\end{matrix} \right)+\left(\begin{matrix}
 0 & L_q\\
 -L_{\bar{q}} & 0
\end{matrix} \right)\left(\begin{matrix}
 0 & L_p\\
 -L_{\bar{p}} & 0
\end{matrix} \right)\\
=& \left(\begin{matrix}
 -(L_pL_{\bar{q}}+L_qL_{\bar{p}}) & 0\\
 0 & -(L_{\bar{p}}L_{q}+L_{\bar{q}}L_{p})
\end{matrix} \right)\\
=&-2(p,q)id.
	\end{align*}

Noticed that every Clifford algebra on even dimensional Euclidean space is universal (\cite[P.25 Corollary(3.6)]{Gilbert}), and
$$dim_{\R}End_{\R}(\mathbb{O}\oplus \mathbb{O})=2^8=dim_{\R}Cl_8.$$
Thus, the map $A$ gives a realization of $Cl_8$, and $\mathbb{O}\oplus \mathbb{O}$ is just the real spinor space $\mathcal{R}_8$ by definition \ref{def:CM}.
\end{proof}
This realization is extremely important for us, and it connects the Clifford analysis and octonionic analytic function theory closely.

\section{The Octonionic analytic functions}
Under the realization in former section, we will give the explicit expression of Dirac operator in this case, which connects octonion Cauchy operator.
And we will see their relationships soon.
\subsection{The Dirac operator and octonion Cauchy operator on $\R^8$}
\begin{definition}[\cite{Gilbert}]
	Under the realization of Clifford algebra $Cl_8$ in the Theorem \ref{realization} , the Dirac operator $D_8$ associated with the Euclidean space $V$ is the first-order differential operator on $\mathcal{C}^1(V, \mathbb{O}\oplus \mathbb{O})$ defined by
	$$D_8\left(\begin{matrix}
		f_1\\
		f_2
	\end{matrix}
	\right)=\sum_{j=0}^7A(e_j)\frac{\partial\ \ }{\partial x_j}\left(\begin{matrix}
		f_1\\
		f_2
	\end{matrix}
	\right)=\sum_{j=0}^7\left(\begin{matrix}
 0 & L_{e_j}\\
 -L_{\overline{e_j}} & 0
\end{matrix} \right)\frac{\partial\ \ }{\partial x_j}\left(\begin{matrix}
		f_1\\
		f_2
	\end{matrix}
	\right)=\left(\begin{matrix}
 0 & D\\
 -\overline{D} & 0
\end{matrix} \right)\left(\begin{matrix}
		f_1\\
		f_2
	\end{matrix}
	\right)$$
Where $f_1,f_2\in \mathcal{C}^1(V, \mathbb{O})$ and $D=\sum_{j=0}^7e_j\frac{\partial\ \ }{\partial x_j}$ is the octonionic Cauchy operator on $\mathcal{C}^1(V, \mathbb{O})$.
\end{definition}
The classical Clifford analysis  \cite{Gilbert,F.S,BDS82} is  to investigate the spinor-valued functions annihilated by Dirac operator, and it has full development, has became an important mathematics branch. It's so clearly $D_8$ annihilated functions is connected with octonionic analytic functions, we give the both definitions formally here.
\begin{definition}\label{Canalytic}
Suppose $\Omega\subset V$ is a domain, $f_1,f_2\in \mathcal{C}^1(V, \mathbb{O})$.
	A function $\left(\begin{matrix}
		f_1\\
		f_2
	\end{matrix}
	\right)\in \mathcal{C}^1(\Omega, \mathbb{O} \oplus \mathbb{O})$ is said to be Clifford analytic on $\Omega$, when $$D_8\left(\begin{matrix}
		f_1\\
		f_2
	\end{matrix}
	\right)=\left(\begin{matrix}
		Df_2\\
		-\overline{D}f_1
	\end{matrix}
	\right)=0$$ on $\Omega$. The set of $\mathcal{R}_8=\mathbb{O}\oplus \mathbb{O}$-valued   Clifford analytic functions in $\Omega$ is denoted by $A(\Omega)$.
\end{definition}

\begin{definition}
  Let $\Omega\subset \R^8$ be a domain, a function $f\in \mathcal{C}^1(\Omega, \O)$ is said to be (left) octonionic analytic on $\Omega$, when
  $$Df=\sum_{j=0}^7\bfe_j\frac{\partial f\ }{\partial x_j}=0$$
  on  $\Omega$, the set of octonionic analytic functions in $\Omega$ is denoted by $A(\Omega, \O)$.
\end{definition}
We give some remarks  here to illustrate the connections between Clifford analytic  and octonionic analytic.
\begin{remark}\label{rem:CandO}
	\begin{enumerate}
		\item For any octonionic analytic function $f$ on $\Omega$, from the Definition \ref{Canalytic}, we know that
		\begin{equation}\label{eq:OandC}
			\left(\begin{matrix}
		0\\
		f
	\end{matrix}
	\right)\in A(\Omega)
		\end{equation}  this is the fundamental evidence why we can study octonionic analytic functions by Clifford analytic functions.
	\item The reader who familiar with Dirac operator and index theorem in differential geometry will be aware that the relationship between $D_8$ and $D$ is connected with the $\mathbb{Z}_2$-grade of spinor spaces and Dirac $\mathcal{D}$-operators and Dirac operators. See \cite[P.207]{Gilbert} or \cite{BGV92} for example.
	\end{enumerate}
\end{remark}
\subsection{A new way to octonionic analysis from Clifford analysis}

Here, with the help of Equation \eqref{eq:OandC}, we can view a octonionic analytic function as a Clifford analytic function. Thus some results on octonionic analytic functions theory can be reformulated. We give some examples here to tells the reader how to realize it.

First, we refer the Cauchy integral theorem of spinor-valued functions, and we will show it how to transfer to octonionic Cauchy integral theorem in detail.

\begin{theorem}[\cite{Gilbert} ]\label{thm:CauchyforC}
If $M$ is a compact, $8$-dimensional, oriented $\mathcal{C}^{\infty}-$manifold in $\Omega$, then for each Clifford analytic function $f$ in $\mathcal{C}^{\infty}(\Omega, \mathcal{R}_8)$, we have
	\begin{equation}
		f(z)=\frac{1}{\omega_8}\int_{\partial M}\frac{(x-z)^{\star}}{\left|x-z\right|^8} \mathrm{d}\sigma(x)f(x).
	\end{equation}
	for each $z$ in the interior of $M$.Where $\mathrm{d}\sigma(x)=A(\eta(x))\mathrm{d}S(x)$ and $\eta(x)$ is the outer unit normal to $\partial M$ at $x$, $\mathrm{d}S(x)$ is the scalar element of surface area on $\partial M$.
\end{theorem}

From this and Equation\eqref{eq:OandC}, we can get the Cauchy integral formula for octonionic analytic functions immediately.
\begin{theorem}
	If $M$ is a compact, $8$-dimensional, oriented $\mathcal{C}^{\infty}-$manifold in $\Omega$, then for each octonionic analytic function $f$ in $\mathcal{C}^{\infty}(\Omega, \mathbb{O})$, we have
	\begin{equation}\label{eq:cauchy}
	f(z)=\frac{1}{\omega_8}\int_{\partial M}\frac{\overline{x-z}}{\left|x-z\right|^8} \left(\mathrm{d}\sigma(x)f(x)\right).
	\end{equation}
	for each $z$ in the interior of $M$. Where $\mathrm{d}\sigma(x)=\eta(x)\mathrm{d}S(x)$ and $\eta(x)\in \mathbb{O}$ is the octonion number determined by the outer unit normal to $\partial M$ at $x$, $\mathrm{d}S(x)$ is the scalar element of surface area on $\partial M$.
\end{theorem}
\begin{proof}
	For any octoionic analytic function $f$ on $\omega$, we have $\left(\begin{matrix}
		0\\
		f
	\end{matrix}
	\right)\in A(\Omega)$, thus for each $z$ in the interior of $M$, by Theorem \ref{thm:CauchyforC}, we have
	
		\begin{align*}
		\left(\begin{matrix}
		0\\
		f
	\end{matrix}
	\right)  &=\frac{1}{\omega_8}\int_{\partial M}\frac{(x-z)^{\star}}{\left|x-z\right|^8} \mathrm{d}\sigma(x)\left(\begin{matrix}
		0\\
		f
	\end{matrix}
	\right)\\
	&=\frac{1}{\omega_8}\int_{\partial M}\frac{1}{\left|x-z\right|^8} \left(\begin{matrix}
 0 & -L_{x-z}\\
 L_{\overline{x-z}} & 0
\end{matrix} \right)\left(\begin{matrix}
 0 & L_{\eta(x)}\\
 -L_{\overline{\eta(x)}} & 0
\end{matrix} \right)\left(\begin{matrix}
		0\\
		f
	\end{matrix}
	\right)\mathrm{d}S(x)\\
	&=\frac{1}{\omega_8}\int_{\partial M}\frac{1}{\left|x-z\right|^8} \left(\begin{matrix}
		0\\
		L_{\overline{x-z}}L_{{\eta(x)}} f
	\end{matrix}
	\right)\mathrm{d}S(x).
	\end{align*}
It's nothing but the equation \eqref{eq:cauchy}, so we done.
\end{proof}
\begin{remark}
\begin{enumerate}
	\item This example shows us how to get octonionic analytic function properties by the   spinor-valued Clifford analytic functions. Actually if the properties only involves the octonionic multiplications and real-linear properties, then we can get these properties by $\mathcal{R}_8$-valued Clifford analytic functions. The deep reason why we can do this way is that the octionic left multiplication operator can gives by Clifford action.
    \item This result has been proved by \cite{LP}, and in this article, the authors get the ``integration by parts" formula.   We point out that it's nothing but the equation (3.22) in \cite[P.102]{Gilbert}.
\end{enumerate}
	\end{remark}
The critical index of subharmonicity of octonionic analytic function has been proved by  \cite{LW} and \cite{AD}, but actually we  can  also get this result from Clifford analytic, and we refer reader to \cite[Chapter 4]{Gilbert} for an excellent introduction.
\begin{theorem}
	Let $\Omega$ be any open set  in $V$. Then
	\begin{enumerate}
		\item whenever $p\ge \dfrac{6}{7}$ and $f$ is a octonionic analytic function on $\Omega$, $x\to\left|f(x)\right|^p$ is subharmonic on $\Omega$.
		\item if $0<p<\dfrac{6}{7}$, there is a octonionic analytic function $f$ such that $x\to\left|f(x)\right|^p$ is not subharmonic anywhere on its domain of definition.
	\end{enumerate}
\end{theorem}
\begin{proof}The first proposition just follows the results of spinor-valued Clifford analytic functions.
	\begin{quote}
	\begin{theorem}\cite[P.108 Corollary 3.39]{Gilbert}
		Let $\Omega$ be any open set  in $V$, whenever $p\ge \dfrac{6}{7}$ and $f$ is a $\mathcal{R}_8$-valued Clifford analytic function on $\Omega$, $x\to\left|f(x)\right|^p$ is subharmonic on $\Omega$.
	\end{theorem}	
	and the fact the norm of $\left(\begin{matrix}
		0\\
		f
	\end{matrix}
	\right)$ is equal to $\left|f\right|$.
	\end{quote}

	 As for the second proposition, just take $f(x)=\dfrac{\overline{x}}{\left|x\right|^8}, x\neq 0$.
\end{proof}
There are still some Clifford analytic results can be transfer into octonionic analytic, for example, Mean-value theorem, Cauchy theorem, Morera's theorem, Maximum Modulus Principle, Weierstrass theorem etc, we don't repeat it so much, we refer to \cite[Chapter2]{Gilbert} and \cite{BDS82} for an introduction to Clifford analytic results and we encourage readers to get octonionic result directly by the method we introduced.
\subsection{Octonionic Hardy space} Review the materials we used in Clifford Hardy space and \eqref{eq:OandC}, the upper half-space Octonionic Hardy space theory can be built without any affords.
\begin{definition}
For any $p>0$, the Hardy space $\HP$ of octonionic analytic functions is defined to be the space of all octonionic analytic $F$ in $\R^8_+$ satisfying :
$$\HN{f}=\sup_{t>0} \left(\int_{\mathbb{R}^7}{\abs{F(t, \underline{x})}^p}\mathrm{d}x\right)^{\frac{1}{p}}<\infty.$$
\end{definition}
Notice that for any $f_1,f_2\in L^p(\R^7, \O), p\ge 1$, recall \eqref{eq:Cauchy} give us :
\begin{align*}
  C\left(\begin{matrix}
		f_1\\
		f_2
	\end{matrix}\right) & = \frac{1}{\omega_8} \int_{\mathbb{R}^{7}} \frac{u-z}{|u-z|^8} \bfe_0 \left(\begin{matrix}
		f_1\\
		f_2
	\end{matrix}\right) \mathrm{d} u\\
&=\frac{1}{\omega_8}\int_{\mathbb{R}^{7}}\frac{1}{\left|u-z\right|^8} \left(\begin{matrix}
 0 & L_{u-z}\\
 -L_{\overline{u-z}} & 0
\end{matrix} \right)\left(\begin{matrix}
 0 & L_1\\
 -L_{1} & 0
\end{matrix} \right)\left(\begin{matrix}
		f_1\\
		f_2
	\end{matrix}
	\right)\mathrm{d}u\\
&=\frac{1}{\omega_8}\int_{\mathbb{R}^{7}}\left(\begin{matrix}
		-\frac{u-z}{|u-z|^8}f_1\\
		-\frac{\overline{u-z}}{|u-z|^8}f_2
	\end{matrix}
	\right)\mathrm{d}u.
\end{align*}
 Thus for  $f\in L^p(\R^7, \O), p\ge 1$, if we define the octonionic Cauchy integral of $f$  by
 $$ C_{\O}(f)=\frac{1}{\omega_8}\int_{\mathbb{R}^{7}}\frac{\overline{z-u}}{|u-z|^8}f(u)\mathrm{d}u.$$
  and $\mathcal{H}_{\O}=-\sum_{j=1}^{7} \bfe_j R_j$, where $R_j$ is the $j$-th Riesz transform as former, then we will have something analogue to Theorem \ref{thm:bdv} and Theorem \ref{thm:Hp}.
 \begin{theorem}\label{thm:obdv}
  Suppose $F \in H^p\left(\mathbb{R}_{+}^n, \O\right), p>\frac{6}{7}$. Then there is a function $f \in$ $L^p\left(\mathbb{R}^{7}, \O\right)$ such that
  \begin{enumerate}
    \item $\lim \limits_{z \rightarrow x  n.t.}$ $F(z)=f(x)$ exists for almost all $x \in \mathbb{R}^{7}$,
    \item $\lim \limits_{t \rightarrow 0}\displaystyle \int_{\mathbb{R}^{7}}|F(x, t)-f(x)|^p \mathrm{d} x=0$.
  \end{enumerate}
\end{theorem}
\begin{proof}
  Just take $\left(\begin{matrix}
		0\\
		F
	\end{matrix}\right)\in H^p\left(\mathbb{R}_{+}^8, \O\oplus \O\right)= H^p\left(\mathbb{R}_{+}^8, \mathcal{R}_8\right)$ in Theorem \ref{thm:bdv}.
\end{proof}

\begin{theorem}\label{thm:oHp}
Suppose that either (i) $1<p<\infty$ and $f \in L^p\left(\mathbb{R}^{7}, \O\right)$, or (ii) $p= 1$ and $f, \mathcal{H}_{\O} f \in L^1\left(\mathbb{R}^{7}, \O\right)$. Then $C_{\O} f \in H^p\left(\mathbb{R}_{+}^8, \O\right)$, and
$$
\lim _{z \rightarrow x, n.t} C_{\O} f(z)=\frac{1}{2}\left(I+\mathcal{H}_{\O} \right) f(x)
$$
for almost all $x \in \mathbb{R}^{7}$.

Conversely, if $1 \leq p<\infty$ and suppose $F \in H^p\left(\mathbb{R}_{+}^8, \O\right)$. Then $F=C_{\O} f$, where $f$ is the almost-everywhere non-tangential limit of $F$ given by Theorem \ref{thm:obdv}.
\end{theorem}
\begin{proof}
  Just take $\left(\begin{matrix}
		0\\
		f
	\end{matrix}\right)\in L^p(\R^7, \O\oplus \O)=L^p(\R^7,\mathcal{R}_8) $ and$\left(\begin{matrix}
		0\\
		F
	\end{matrix}\right)\in H^p\left(\mathbb{R}_{+}^8, \O\oplus \O\right)=H^p\left(\mathbb{R}_{+}^8, \mathcal{R}_8\right)$ in Theorem  \ref{thm:Hp}.
\end{proof}
Theorem \ref{thm:oHp} says that for $1<p<\infty, H^p\left(\mathbb{R}_{+}^8, \O\right)$ is precisely the image of $L^p\left(\mathbb{R}^{7}, \O\right)$ under the octonionic Cauchy integral operator $C_{\O}$, while $H^1\left(\mathbb{R}_{+}^8, \O\right)$ is the image under $C_{\O}$ of the set of all $f \in$ $L^1\left(\mathbb{R}^{7}, {\O}\right)$ for which $\mathcal{H}_{\O} f \in L^1\left(\mathbb{R}^{7}, {\O}\right)$.

Moreover, for $1 \leq p<\infty$, $f \in L^p\left(\mathbb{R}^{n-1}, \O\right)$ arises as the non-tangential boundary value of an $H^p\left(\mathbb{R}_{+}^n, \O\right)$ function if and only if $f=\mathcal{H}_{\O}  f \in L^p\left(\mathbb{R}^{7}, \O\right)$.

\section{Proof of main theorems}
Now we turn to the proof of Theorem \ref{thm:main1}:
\begin{proof}[ Proof of Theorem \ref{thm:main1}]
  First we proof for the spinor space $\mathcal{R}_8=\O\oplus \O$, Theorem \ref{thm:main1} holds.
  We take $\eta=\bfe_0, \mathfrak{H}_0=\O\left(\begin{matrix}
		1\\
		1
	\end{matrix}\right)=\left\{\left(\begin{matrix}
		p\\
		p
	\end{matrix}\right): p\in \O\right\},
$ then we have:
$$\eta \mathfrak{H}_0=\left\{\left(\begin{matrix}
		0&1\\
		-1&0
	\end{matrix}\right)\left(\begin{matrix}
		p\\
		p
	\end{matrix}\right): p\in \O\right\}=\left\{\left(\begin{matrix}
		p\\
		-p
	\end{matrix}\right): p\in \O\right\}=\O\left(\begin{matrix}
		1\\
		-1
	\end{matrix}\right).$$
So we have
\begin{equation}\label{eq:subspace}
  \O\oplus\O=\mathfrak{H}_0\oplus \eta \mathfrak{H}_0.
\end{equation}

 Before check the Cauchy integral operator $C$ is an isomorphism from $L^p(\R^7, \mathfrak{H}_0 )$ onto $H^p(\R^8_+,\O\oplus\O)$ for all $p>1$, we need some explains the meaning of isomorphism.

 For any $f\in L^p(\R^7, \mathfrak{H}_0 ),p>1 $, we know $C(f)\in H^p(\R^8_+,\O\oplus\O) $ by Theorem\ref{thm:Hp}, and $C(f)$ has a non-tangent limits $\dfrac{1}{2}(f(x)+\bfe_0\mathcal{H}f(x))\in L^p(\R^7, \O\oplus\O ) $ for almost $x\in\R^7$, so we must check the projection of $(f+\bfe_0\mathcal{H}f)$ onto $\mathfrak{H}_0$  is just $f\in L^p(\R^7, \mathfrak{H}_0 )$.

 Conversely, for any $F\in H^p(\R^8_+,\O\oplus\O), p>1 $ it has a non-tangent limits
 $f\in L^p(\R^7, \O\oplus\O )$, which satisfies $f=\bfe_0\mathcal{H}f$.  It has a decomposition
 \begin{equation*}
   f=g\left(\begin{matrix}
		1\\
		1
	\end{matrix}\right)+h\left(\begin{matrix}
		1\\
		-1
	\end{matrix}\right)
 \end{equation*}
 $g,h\in L^p(\R^7,\O)$ by \eqref{eq:subspace}. So  we also need   check
  \begin{equation}\label{eq:task2}
     f=g\left(\begin{matrix}
		1\\
		1
	\end{matrix}\right)+\bfe_0\mathcal{H}\left[g\left(\begin{matrix}
		1\\
		1
	\end{matrix}\right)\right].
  \end{equation}

 But all these just follow the  action of $Cl_8$ on $\O\oplus \O$, more precisely:
For $p>1$ and  $f=g\left(\begin{matrix}
		1\\
		1
	\end{matrix}\right)\in L^p(\R^7, \mathfrak{H}_0 ), g\in L^p(\R^7, \O) $ we have
\begin{align*}
  f+\bfe_0\mathcal{H}f & =\left(\begin{matrix}
		g\\
		g
	\end{matrix}\right)+\sum_{j=1}^{7}\left(\begin{matrix}
		0&1\\
		-1&0
	\end{matrix}\right)\left(\begin{matrix}
		0&L_{\bfe_j}\\
		L_{\bfe_j}&0
	\end{matrix}\right)R_j\left(\begin{matrix}
		g\\
		g
	\end{matrix}\right) \\
&=\left(\begin{matrix}
		g\\
		g
	\end{matrix}\right)+\sum_{j=1}^{7}\left(\begin{matrix}
		L_{\bfe_j}&0\\
		0&-L_{\bfe_j}
	\end{matrix}\right)\left(\begin{matrix}
		R_jg\\
		R_jg
	\end{matrix}\right)\\
&=\left(\begin{matrix}
		g\\
		g
	\end{matrix}\right)+\sum_{j=1}^{7}L_{\bfe_j}R_jg\left(\begin{matrix}
		1\\
		-1
	\end{matrix}\right)=g\left(\begin{matrix}
		1\\
		1
	\end{matrix}\right)-\mathcal{H}_{\O}g\left(\begin{matrix}
		1\\
		-1
	\end{matrix}\right)
\end{align*}
So the $\mathfrak{H}_0$ part of $f+\bfe_0\mathcal{H}f$ is just $f$. And follow the result above, the equation \eqref{eq:task2} is noting but $h=-\mathcal{H}_{\O}g$.

By $f=\bfe_0\mathcal{H}f$, we can get
\begin{align}\label{eq:f=hf}
  g\left(\begin{matrix}
		1\\
		1
	\end{matrix}\right)+h\left(\begin{matrix}
		1\\
		-1
	\end{matrix}\right)=&\sum_{j=1}^{7}\left(\begin{matrix}
		0&1\\
		-1&0
	\end{matrix}\right)\left(\begin{matrix}
		0&L_{\bfe_j}\\
		L_{\bfe_j}&0
	\end{matrix}\right)R_j\left[g\left(\begin{matrix}
		1\\
		1
	\end{matrix}\right)+h\left(\begin{matrix}
		1\\
		-1
	\end{matrix}\right)\right]\nonumber\\
 =&\sum_{j=1}^{7}L_{\bfe_j}R_jg\left(\begin{matrix}
		1\\
		-1
	\end{matrix}\right)+\sum_{j=1}^{7}L_{\bfe_j}R_jh\left(\begin{matrix}
		1\\
		1
	\end{matrix}\right)=-\mathcal{H}_{\O}h\left(\begin{matrix}
		1\\
		1
	\end{matrix}\right)-\mathcal{H}_{\O}g\left(\begin{matrix}
		1\\
		-1
	\end{matrix}\right).
\end{align}
So we complete the proof of (2) in Theorem \ref{thm:main1} .

For (3), we denote $Tan$  the projection of $\O\oplus \O$ onto $\mathfrak{H}_0$, and $F^+\in L^p(\R^7, \O\oplus \O), p>1$ be the almost-everywhere non-tangential limits of a Clifford Hardy function $F\in H^p(\R^8_+, \O\oplus \O)$, so the boundary operator mapping from $H^p(\R^8_+, \O\oplus \O)$ onto $L^p(\R^7, \mathfrak{H}_0 )$ is just $F\to Tan(F^+)$. By theorem \ref{thm:bdv}, we have
$$\| F^+\|_{L^p(\R^7,\O\oplus\O)}=\lim_{t\to 0}\left(\int_{\R^7}\abs{F(x,t)}^p\mathrm{d}x\right)^{\frac{1}{p}}\le \sup_{t> 0}\left(\int_{\R^7}\abs{F(x,t)}^p\mathrm{d}x\right)^{\frac{1}{p}}=\|F\|_{H^p(\R^8_+, \O\oplus \O)}$$
Thus we have
$$\|Tan ( F^+)\|_{L^p(\R^7,\mathfrak{H}_0)}\le \|F^+\|_{L^p(\R^7,\O\oplus\O)}\le \|F\|_{H^p(\R^8_+, \O\oplus \O)}.$$
which implies (3) in Theorem \ref{thm:main1}, so we complete our proof in $H^p(\R^8_+,\mathcal{R}_8)$ case.

In the general cases, for any Clifford module $\mathfrak{H}$,  since $Cl_8$ is a simple algebra, consequently, $\mathfrak{H}$ is isomorphic to a direct sum of $\mathcal{R}_8$. So  theorem \ref{thm:main1} still holds for any  Clifford module $\mathfrak{H}$, we have complete the conjecture on $\R^8$ totally.

\end{proof}

\begin{remark}
  In the \eqref{eq:f=hf}, we get not only $h=-\mathcal{H}_{\O}g$, but also $g=-\mathcal{H}_{\O}h$. But it does not surprise us, because in general we have $$\mathcal{H}_{\O}^2g=g,\quad \forall g\in L^p(\R^7,\O).$$
  This fact based on some basic properties of Riesz transforms $R_j$, for example
  $$R_jR_kg=R_kR_jg, $$ and
  $$\sum_{j=1}^{7}R_j^2g=-g,\forall g\in L^p(\R^7,\O).$$
   All these properties can be obtained by Fourier transform easily, we refer  readers to \cite[P.324]{GTM249}  for a detailed introduction.
\end{remark}
At last of the section, we give a proof of Theorem \ref{thm:main2}.
\begin{proof}[Proof of Theorem \ref{thm:main2}]
First we prove that there is a real-valued Schwartz function $f\in \mathcal{S}(\R^7)$, such that
  \begin{equation*}
       R_jf(0)=\frac{2}{\omega_n} \int_{\mathbb{R}^{7}} \frac{-u_j}{|u|^8} f(u) \mathrm{d} u=\delta_{j1},\quad j=1,2,\cdots,7.
  \end{equation*}
Actually we can take $f(x)=c x_1e^{-\abs{x}^2}$ and the constant $c$ is selected to satisfy
$R_1f(0)=1$.

  Now, suppose we have two proper subspaces $\mathfrak{H}_0, \mathfrak{H}_1$ of $\O$ satisfy all three statements in  Theorem \ref{thm:main2}, and
  $$\mathfrak{H}_0=span_{\R}\{\xi_1,\xi_2,\cdots,\xi_m\}.$$
  Where $\{\xi_1,\xi_2,\cdots,\xi_m\}$ forms a orthogonal basis of $\mathfrak{H}_0$.

  Take $g=f\xi_1\in L^p(\R^7,\mathfrak{H}_0)$, we know that $C_{\O}(f)$ has a non-tangential limits
  $$C_{\O}(f)^+=\dfrac{1}{2}\left(g+\mathcal{H}_{\O}g\right),$$
  by Theorem \ref{thm:oHp}.
  And according  (2) in the Theorem \ref{thm:main2}, we must have $$Tan(C_{\O}(f)^+)=g=f\xi_1,$$
  so $\mathcal{H}_{\O}g\in L^p(\R^7,\mathfrak{H}_1)$. While $\mathcal{H}_{\O}g(0)=\bfe_1\xi_1\in \mathfrak{H}_1$. Under the same reason, we have
  $$\bfe_j\xi_1\in \mathfrak{H}_1, \quad \forall j=1,2,\cdots,7.$$
  But we have
  $$(\bfe_i\xi_1,\bfe_j\xi_1)=((\bfe_i\xi_1)\overline{\xi_1}, \bfe_j)=(\bfe_i,\bfe_j)=\delta_{ij}, \quad \forall i,j=1,2,\cdots ,7.$$
  That means $dim_{\R}\mathfrak{H}_1\ge 7$, so we have $\mathfrak{H}_0=span_{\R}\{\mathbf{p}\}$, while $$\mathfrak{H}_1=span_{\R}\{\bfe_1\mathbf{p},\bfe_2\mathbf{p},\ldots,\bfe_7\mathbf{p}\}.$$

   Now for a $F \in H^p(\R^8_+ \O), p>1 $ it has a non-tangent limits
 $F^+\in L^p(\R^7, \O)$, which satisfies $F^+=\mathcal{H}_{\O}F^+$.
 Suppose $$F^+=\sum_{j=0}^7f_j\bfe_j\mathbf{p}$$ where $f_j\in L^p(\R^7,\R), j=1, 2, \ldots,7$.
 So we have $$Tan(F^+)=h:=f_0\mathbf{p}\in L^p(\R^7,\mathfrak{H}_0),$$

 Now if (2) in the Theorem \ref{thm:main2} holds,  we must have
 $$F^+=h+\mathcal{H}_{\O}h=h-\sum_{j=1}^{7}\bfe_jR_jh=f_0\mathbf{p}-\sum_{j=1}^{7}R_jf_0\bfe_j\mathbf{p}.$$
 So we have
 \begin{equation}\label{eq:SW}
   f_j=-R_j(f_0), j=1,2,\ldots,7.
 \end{equation}

 We turn our direction to  take a $F \in H^p(\R^8_+ \O), p>1 $ such that \eqref{eq:SW} doesn't hold, consequently,  (2) in the Theorem \ref{thm:main2} doesn't hold.

Take $g\in L^p(\R^7,\R)$ and $$F^+:=g\bfe_1\mathbf{p}+\mathcal{H}_{\O}g\bfe_1\mathbf{p},$$ then we have $F^+=\mathcal{H}_{\O}F^+$,
Theorem \ref{thm:oHp} tells us it's a non-tangential limits of a octonion Hardy function $F\in H^p(\R^8_+,\O)$.
We claim that for this $F^+=\sum\limits_{j=0}^7f_j\bfe_j\mathbf{p}$ is what we want, it means \eqref{eq:SW} doesn't hold for $F^+$.


As a fact, $$F^+:=g\bfe_1\mathbf{p}-\sum_{j=1}^{7}R_j\bfe_j (g\bfe_1\mathbf{p})$$
So we have :
\begin{equation}\label{eq:f+}
 F^+(\overline{p})=g\bfe_1\mathbf{p}(\overline{\mathbf{p}})-\sum_{j=1}^{7}R_jg (\bfe_j(\bfe_1\mathbf{p}))\overline{\mathbf{p}}=R_1g+g\bfe_1-\sum_{j=2}^{7}g_j \bfe_j
\end{equation}
And \eqref{eq:f+} just follows that
\begin{align*}
  ((\bfe_j(\bfe_1\mathbf{p}))\overline{\mathbf{p}},\bfe_0) & =-(\bfe_1\mathbf{p},\bfe_j\mathbf{p})=-\delta_{1j},j=1,\ldots,7 \\
  ((\bfe_j(\bfe_1\mathbf{p}))\overline{\mathbf{p}},\bfe_1) & =((\bfe_j(\bfe_1\mathbf{p})),\bfe_1\mathbf{p})=0,j=1,\ldots,7.
\end{align*}

So we have $g_0=R_1g$ and $g_1=g$. Thus if \eqref{eq:SW} holds, then we have $$g=-R_1^2g,\quad \forall g\in L^p(\R^7,\R)$$
which is impossible. So we complete our proof.
\end{proof}

\begin{remark}
We give some remarks on  equation \eqref{eq:SW} .

Now we suppose  equation \eqref{eq:SW} holds,
let
   $$u_j(t, x)=P_t*f_j, \quad  j=0,1,\ldots, 7, $$
   follow the \cite[P.332]{GTM249} , we know that $\mathbf{F}=(u_0, u_1, \ldots, u_7)$ satisfies the following system of generalized Cauchy-Riemann equations:
\begin{align*}
  \sum_{j=0}^{7}\frac{\partial u_j}{\partial x_j}(t,x) =0  \\
  \frac{\partial u_j}{\partial x_k}(t,x)=\frac{\partial u_k}{\partial x_j}(t,x), 0\le j\neq k\le 7,
\end{align*}
here $\dfrac{\partial \ }{\partial x_0}=\dfrac{\partial \ }{\partial t}$. And this implies $\mathbf{F}\in H^p(\R^8_+)$ , where$H^p(\R^8_+)$ denotes the  $H^p$ space of conjugate harmonic functions in Stein-Weiss sense\cite[P.236]{SW71}.

What's more the Cauchy integral of $f_0$ is just
$$F=C_{\O}(f_0)=\frac{1}{2}\sum_{j=0}^{7}u_j\overline{\bfe_j}. $$
A classical result \cite{LP2} on octonionic analytic functions says
that $F$ is not only left octonionic analytic, but also right octonionic analytic.
But we know, not all left octonionic analytic function is either  right octonionic analytic. Thus, from this  point view, we can also get that the equation \eqref{eq:SW} doesn't hold generally.

\end{remark}

%


\begin{thebibliography}{\;\;\, Roman}

\bibitem[Atiyah64]{Atiyah}
 Atiyah M. F.,  Bott R.,  Shapiro A.,\emph{Clifford Modules}, Topology 3 , no. suppl, suppl. 1,(1964): 3-38.

\bibitem[Baez05]{Baez}
Baez, J.C. , \emph{The octonions.} Bull. Amer. Math. Soc. (N.S.) \textbf{39} (2002), no. {2}, 145–205.

\bibitem[BGV92]{BGV92}
Berline, N. , Getzler, E. , Vergne, M., \emph{Heat kernels and Dirac operators.}
Grundlehren der mathematischen Wissenschaften  \textbf{298}. Springer-Verlag, Berlin, 1992. viii+369 pp. ISBN: 3-540-53340-0.

\bibitem[BDS82]{BDS82}
Brackx, F., Delanghe, R., Sommen, F., \emph{Clifford analysis.} Research Notes in Mathematics \textbf{76}. Pitman (Advanced Publishing Program), Boston, MA, 1982. x+308 pp. ISBN: 0-273-08535-2.


\bibitem[CS03]{CS03}
Conway J. H.,  Smith, D. A. ,\emph{On quaternions and octonions: their geometry, arithmetic, and symmetry.} A K Peters, Ltd., Natick, MA, 2003. xii+159 pp. ISBN: 1-56881-134-9.

\bibitem[DSS92]{F.S}
Delanghe R. , Sommen F., Souček V. , \emph{Clifford algebra and spinor valued functions: a function theory for the Dirac operator}, Springer Netherlands,
Kluwer Academic, volume \textbf{53}, 1992. XVII: 485pp. ISBN: 978-94-011-2922-0.

\bibitem[Gilbert91]{Gilbert}
Gilbert, J.E. ,  Murray, M. A. M. , \emph{Clifford Algebras and Dirac Operators in Harmonic Analysis},Cambridge Studies in Advanced Mathematics, \textbf{26}. Cambridge University Press, Cambridge, 1991. viii+334 pp. ISBN: 0-521-34654-1.

\bibitem[GTM249]{GTM249}
Grafakos L., \emph{Classical Fourier analysis.} Third edition. Graduate Texts in Mathematics \textbf{249}. Springer, New York, 2014. xviii+638 pp. ISBN: 978-1-4939-1193-6; 978-1-4939-1194-3.


\bibitem[Harvey90]{Harvey90}
Harvey, F. R., \emph{Spinors and calibrations.}
Perspectives in Mathematics \textbf{9}. Academic Press Inc., Boston, MA, 1990. xiv+323 pp. ISBN: 0-12-329650-1.

\bibitem[HLR21]{HLR21}
 Huo Q. H. , Li Y. , Ren G. B. , \emph{Classification of left octonionic modules.}  Adv. Appl. Clifford Algebr, \textbf{31} (2021), no. 1, Paper No. 11, 14 pp.

\bibitem[AD06]{AD}
Kheyfits A.,  Tepper D., \emph{Subharmonicity of powers of octonion-valued monogenic functions and some applications}. Bull. Belg. Math. Soc. Simon Stevin, \textbf{13}(2006), 609–617.

\bibitem[LP02]{LP}
Li X. M., Peng L. Z., \emph{The Cauchy integral formulas on the octonions.}
Bull. Belg. Math. Soc. Simon Stevin , \textbf{9} (2002), no. 1, 47–64.

\bibitem[LP2]{LP2}
Li  X. M, Peng, L. Z. , \emph{On Stein-Weiss conjugate harmonic function and octonion analytic function.} Approx. Theory Appl. (N.S.) \textbf{16} (2000), no. 2, 28–36.

\bibitem[LW14]{LW}
Li X. M., Wang J. X., \emph{Orthogonal invariance of the Dirac operator and the critical index of subharmonicity for octonionic analytic functions.} Adv. Appl. Clifford Algebr. \textbf{24} (2014), no. 1, 141–149.

\bibitem[SW71]{SW71}
Stein, E. M., Weiss G., \emph{Introduction to Fourier analysis on Euclidean spaces.}
Princeton Mathematical Series, No. \textbf{32}. Princeton University Press, Princeton, N.J., 1971. x+297 pp.



\end{thebibliography}
\end{document}